\documentclass{amsart}

\usepackage[all]{xy}

\theoremstyle{plain}
\newtheorem{thm}{Theorem}[section]
\newtheorem{prop}[thm]{Proposition}
\newtheorem{lemma}[thm]{Lemma}

\theoremstyle{definition}

\theoremstyle{remark}
\newtheorem{rem}[thm]{Remark}

\begin{document}

\title[Universal deformation rings related to $S_5$]{Universal deformation rings for the 
symmetric group $S_5$ and one of its double covers}

\author{Frauke M. Bleher}
\address{F.B.: Department of Mathematics\\University of Iowa\\
Iowa City, IA 52242-1419, U.S.A.}
\email{fbleher@math.uiowa.edu}
\thanks{The first author was supported in part by  
NSF Grant DMS06-51332  and NSA Grant H98230-06-1-0021.}
\author{Jennifer B. Froelich}
\address{J.F.: Department of Mathematics and Computer Science\\
Dickinson College\\ Carlisle, PA 17013, U.S.A.}
\email{froelicj@dickinson.edu}

\subjclass[2000]{Primary 20C20; Secondary 16G20}
\keywords{Universal deformation rings, dihedral defect groups, quaternion defect groups}

\begin{abstract}
Let $S_5$ denote the symmetric group on $5$ letters,
and let $\hat{S}_5$ denote a non-trivial double cover of $S_5$
whose  Sylow $2$-subgroups are generalized quaternion.  
We determine the universal deformation rings
$R(S_5,V)$ and $R(\hat{S}_5,V)$ of each mod $2$ representation $V$ of $S_5$
that belongs to the principal $2$-modular block of $S_5$ and
whose stable endomorphism ring is given by scalars when it is inflated to $\hat{S}_5$.
We show that for these $V$, a question raised by the first author and Chinburg concerning 
the relation of the universal deformation ring of $V$ to the Sylow $2$-subgroups of $S_5$
and $\hat{S}_5$, respectively, has an affirmative answer. 
\end{abstract}

\maketitle


\section{Introduction}
\label{s:intro}

Let $k$ be an algebraically closed field of characteristic $p>0$ and let $W=W(k)$ be the ring of infinite 
Witt vectors over $k$. Let $G$ be a finite group, and suppose $V$ is a finitely generated $kG$-module.
It was proved in \cite{bc} that if the stable endomorphism ring $\underline{\mathrm{End}}_{kG}(V)$ is 
one-dimensional over $k$ then $V$ has a universal deformation ring $R(G,V)$. The ring $R(G,V)$ is 
universal with respect to deformations of $V$ over complete local commutative Noetherian rings with 
residue field $k$ (for details, see \S\ref{s:prelim}). 
In \cite{bl,diloc,3sim,bc,bc5,llosent}, the isomorphism types of $R(G,V)$ have been determined for $V$ 
belonging to cyclic blocks, respectively to various tame blocks with dihedral defect groups.
In the present paper, we will consider the principal $2$-modular blocks of the symmetric group
$S_5$ and one of its double covers $\hat{S}_5$ whose Sylow $2$-subgroups are generalized
quaternion. One of the main goals is to investigate how the universal deformation rings change
when inflating modules from $S_5$ to $\hat{S}_5$.
The key tools used to determine 
the universal deformation rings in all the above cases
have been results from modular and ordinary representation theory due to 
Brauer, Erdmann \cite{erd}, Linckelmann \cite{linckel1}, Carlson-Th\'{e}venaz \cite{carl2}, 
and others.

The main motivation for studying universal deformation rings for finite groups is that this case helps
one understand ring theoretic properties of universal deformation rings for profinite groups $\Gamma$.
The latter have become an important tool in number theory, in particular if $\Gamma$ is a
profinite Galois group (see e.g. \cite{cornell} and its references).
In \cite{lendesmit}, de Smit and Lenstra showed
that if $\Gamma$ is an arbitrary profinite group and $V$ is a finite
dimensional vector space over $k$ with a continuous $\Gamma$-action which has a universal
deformation ring  $R(\Gamma,V)$, then $R(\Gamma,V)$ is the inverse limit of the universal 
deformation rings $R(G,V)$ when $G$ ranges over all finite discrete quotients of $\Gamma$ through 
which the $\Gamma$-action on $V$ factors. Thus to answer questions about the ring structure of 
$R(\Gamma,V)$, it is natural to first consider the case when $\Gamma=G$ is finite.

Suppose now that the characteristic of $k$ is $2$ and that 
$S_5$ and $\hat{S}_5$ are as above. The Sylow $2$-subgroups of $S_5$ are dihedral
groups of order $8$, whereas the Sylow $2$-subgroups of $\hat{S}_5$ are generalized
quaternion groups of order $16$. The center $Z$ of $\hat{S}_5$ has 2 elements and
$\hat{S}_5/Z\cong S_5$.  Since $Z$ acts
trivially on the simple $k\hat{S}_5$-modules, they are all inflated from simple $kS_5$-modules.
Moreover, the simple modules belonging to the principal block of $k\hat{S}_5$ are inflated from the
simple $kS_5$-modules belonging to the principal block of $kS_5$. 
There are precisely two isomorphism classes of simple $kS_5$-modules belonging to the 
principal block of $kS_5$. They are represented by the trivial simple module $T_0$ and a 
$4$-dimensional simple module $T_1$.

Our main result is as follows, where $W[\mathbb{Z}/2]$ denotes the group ring over $W$ of the
cyclic group $\mathbb{Z}/2$.

\begin{thm}
\label{thm:supermain}
Let $B$ $($resp. $\hat{B}$$)$ be the principal block of $kS_5$ $($resp. $k\hat{S}_5$$)$.
Let $V$ be an indecomposable $kS_5$-module belonging to $B$, and denote its inflation to 
$k\hat{S}_5$ also by $V$, so $V$ belongs to both $B$ and $\hat{B}$.
\begin{enumerate}
\item[(a)] Then
$\underline{\mathrm{End}}_{k\hat{S}_5}(V) \cong k$ if and only if  $\mathrm{End}_{kS_5}(V)\cong k$. 
Moreover, we have $\mathrm{End}_{kS_5}(V)\cong k$
if and only if $V$ is either isomorphic to $T_0$ or a uniserial $kS_5$-module whose
radical series length is at most $3$ and which is a submodule or a quotient module of the projective
$kS_5$-cover of $T_1$.
\vspace{1ex}
\item[(b)] Suppose $\mathrm{End}_{kS_5}(V)\cong k$.
\begin{enumerate}
\item[(i)] If $V\cong T_0$, then $R(S_5,V)\cong W[\mathbb{Z}/2]\cong R(\hat{S}_5,V)$.
\item[(ii)] If $V\cong T_1$, then $R(S_5,V)\cong k$ and $R(\hat{S}_5,V)\cong W$.
\item[(iii)] If the radical series length of $V$ is $2$, then $R(S_5,V)\cong W[\mathbb{Z}/2]\cong
R(\hat{S}_5,V)$.
\item[(iv)] If  the radical series length of $V$ is $3$, then $R(S_5,V)\cong W[[t]]/(t^2,2t)$ and
$R(\hat{S}_5,V)\cong W[[t]]/(t^3-2t)$.
\end{enumerate}
\end{enumerate}
\end{thm}

In particular, the universal deformation rings $R(\hat{S}_5,V)$ are all complete 
intersection rings, whereas for $V$ as in part (iv), $R(S_5,V)$ is not a complete intersection. 
Note that for all cases (i)-(iv),
$R(S_5,V)$ (resp. $R(\hat{S}_5,V)$) is isomorphic to a subquotient ring of $WD_8$ (resp. 
$WQ_{16}$) when $D_8$ is a dihedral group of order $8$ (resp. $Q_{16}$ is a generalized 
quaternion group of order $16$).
In particular, this gives a positive answer in these cases  to a 
question raised by the first author and Chinburg in \cite[Question 1.1]{bc} whether the universal 
deformation ring of a representation of a finite group whose stable endomorphism ring is isomorphic to
$k$ is always isomorphic to a subquotient ring of the group ring over $W$ of a defect group of the 
modular block  associated to the representation.

The paper is organized as follows. In \S \ref{s:prelim}, we give some background on universal deformation rings. In \S \ref{s:ks5}, we state properties of the principal $2$-modular
block $B$ (resp. $\hat{B}$) of $S_5$ (resp. $\hat{S}_5$) 
and prove part (a) of Theorem \ref{thm:supermain}.
In \S \ref{s:udr}, we determine the universal deformation rings
of the $B$-modules whose endomorphism rings are isomorphic to $k$ and 
of their inflations to $\hat{B}$ and prove part (b) of Theorem \ref{thm:supermain}.
In \S\ref{s:append}, we list the ordinary and the $2$-modular character table of $\hat{S}_5$.

This paper is based on the Ph.D. thesis of the second author under the supervision
of the first author \cite{froelich}. We would like to thank the referee for helpful comments.


\section{Preliminaries}
\label{s:prelim}

Let $k$ be an algebraically closed field of characteristic $p>0$, let $W$ be the ring of infinite Witt 
vectors over $k$ and let $F$ be the fraction field of $W$. Let ${\mathcal{C}}$ be the category of 
all complete local commutative Noetherian rings with residue field $k$. The morphisms in 
${\mathcal{C}}$ are continuous $W$-algebra homomorphisms which induce the identity map on $k$. 

Suppose $G$ is a finite group and $V$ is a finitely generated $kG$-module. 
A lift of $V$ over an object $R$ in ${\mathcal{C}}$ is a finitely generated $RG$-module $M$ which
is free over $R$ together with a $kG$-module isomorphsim $\phi:k\otimes_R M\to V$. Two lifts 
$(M,\phi)$ and $(M',\phi')$ of $V$ over $R$ are isomorphic if there is an $RG$-module isomorphism 
$f:M\to M'$ such that $\phi'\circ(k\otimes_R f) = \phi$.
The isomorphism class of a lift of $V$ 
over $R$ is called a deformation of $V$ over $R$, and the set of such deformations is denoted by 
$\mathrm{Def}_G(V,R)$. The deformation functor ${F}_V:{\mathcal{C}} \to \mathrm{Sets}$
is defined to be the covariant functor which sends an object $R$ in ${\mathcal{C}}$ to 
$\mathrm{Def}_G(V,R)$.

If there exists an object $R(G,V)$ in ${\mathcal{C}}$ and a lift $(U(G,V),\phi_U)$ of $V$ over 
$R(G,V)$ 
such that for each $R$ in ${\mathcal{C}}$ and for each lift $(M,\phi)$ of $V$ over $R$ there is a unique 
morphism $\alpha:R(G,V)\to R$ in ${\mathcal{C}}$ such that 
$(M,\phi)$ is isomorphic to $(R\otimes_{R(G,V),\alpha}U(G,V),\phi_U)$, 
then $R(G,V)$ is called the universal deformation ring of $V$ and the isomorphism class
of the lift $(U(G,V),\phi_U)$ is called the universal deformation of $V$. In other words, $R(G,V)$ represents
the functor ${F}_V$ in the sense that ${F}_V$ is naturally isomorphic to 
$\mathrm{Hom}_{{\mathcal{C}}}(R(G,V),-)$.  For more information on deformation rings see 
\cite{lendesmit} and \cite{maz1}.

Suppose $V$ has a universal deformation ring $R(G,V)$ and a universal lift $(U(G,V),\phi_U)$
over $R(G,V)$ that represents the universal deformation of $V$. 
Then we call $\overline{R}=R(G,V)/p\,R(G,V)$  the universal mod $p$ deformation ring
of $V$ and we call the isomorphism class of the lift $(\overline{R}\otimes_{R(G,V)}U(G,V),\phi_U)$ the 
universal mod $p$ deformation of $V$. Note that $\overline{R}$ represents the restriction
of the deformation functor $F_V$ to the full subcategory of $\mathcal{C}$ of objects that are
$k$-algebras.

The following two results were proved in \cite{bc}.
Here $\Omega$ denotes the syzygy, or Heller, operator for $kG$ (see for example \cite[\S 20]{alp}).

\begin{prop}
\label{prop:stablend}
{\rm (\cite[Prop. 2.1]{bc}).}
Suppose $V$ is a finitely generated $kG$-module whose stable endomorphism ring 
$\underline{\mathrm{End}}_{kG}(V)$ is isomorphic to $k$.  Then $V$ has  a universal 
deformation ring $R(G,V)$.
\end{prop}

\begin{lemma} 
\label{lem:defhelp}
{\rm (\cite[Cor. 2.5]{bc}).}
Let $V$ be a finitely generated $kG$-module with $\underline{\mathrm{End}}_{kG}(V)\cong k$.
Then $\underline{\mathrm{End}}_{kG}(\Omega(V))\cong k$, and $R(G,V)$ and $R(G,\Omega(V))$ 
are isomorphic.
\end{lemma}


\section{The principal $2$-modular blocks of $S_5$ and $\hat{S}_5$}
\label{s:ks5}

Let $k$ be an algebraically closed field of characteristic $2$, let $W$ be the ring of infinite
Witt vectors over $k$ and let $F$ be the fraction field of $W$. 

Let $B$ (resp. $\hat{B}$) be the principal block of $kS_5$ (resp. of $k\hat{S}_5$). 
Then the defect groups of $B$ (resp. of $\hat{B}$) are dihedral groups of order $8$
(resp. generalized quaternion groups of order $16$). Looking at the ordinary and the $2$-modular
character table of $\hat{S}_5$ (see \S\ref{s:append}), we see that the decomposition 
matrix for $B$ (resp. for $\hat{B}$) is as in Figure \ref{fig:decomp}. 
\begin{figure}[ht] \hrule \caption{\label{fig:decomp} The decomposition matrix for $B$
(resp. for $\hat{B}$).}
$$\begin{array}{cc}
&\begin{array}{c@{}c}\varphi_0\,&\,\varphi_1\end{array}\\[1ex]
\begin{array}{c}\chi_1\\ \chi_2\\ \chi_3 \\ \chi_4\\ \chi_5\end{array} &
\left[\begin{array}{cc}1&0\\1&0\\1&1\\1&1\\2&1\end{array}\right]\end{array}
\qquad\qquad \left(\mbox{resp.}\quad
\begin{array}{cc}
&\begin{array}{c@{}c}\varphi_0\,&\,\varphi_1\end{array}\\[1ex]
\begin{array}{c}\psi_1\\ \psi_2\\ \psi_3 \\ \psi_4\\ \psi_5\\ \psi_6\\ \psi_7\\ \psi_8\end{array} &
\left[\begin{array}{cc}1&0\\1&0\\1&1\\1&1\\0&1\\2&1\\2&1\\2&1\end{array}\right]
\end{array}
\right).$$
\vspace{2ex}
\hrule
\end{figure}

\begin{rem}
\label{rem:ordinary}
The field $F$ is a splitting field for $S_5$. It follows from the ordinary character table
of $\hat{S}_5$ in Figure \ref{fig:ordchar2s5} and from \cite[Thm. A]{feit} that the Schur indices
of all irreducible characters of $\hat{S}_5$ with respect to $F$ are $1$. Hence 
the characters
$\psi_1,\psi_2,\ldots,\psi_6$ (resp. $\psi_7,\psi_8$) correspond to irreducible representations
of $\hat{S}_5$ which are realizable over $F$ (resp. over $F(\sqrt{2})$). Moreover,
$\psi_7,\psi_8$ are conjugate under the action of the Galois group of  $F(\sqrt{2})$ over $F$.
Hence the characters of  the irreducible representations of $\hat{S}_5$ over $F$ which belong to
$\hat{B}$ are
$$\psi_1,\psi_2,\ldots,\psi_5,\psi_6,\psi_7+\psi_8.$$
If $V_6$ (resp. $V_{78}$) is the $F\hat{S}_5$-module whose character is $\psi_6$ (resp.
$\psi_7+\psi_8$), then $\mathrm{End}_{F\hat{S}_5}(V_6)\cong F$ and 
$\mathrm{End}_{F\hat{S}_5}(V_{78})\cong F(\sqrt{2})$.
\end{rem}

Using the decomposition matrices in Figure \ref{fig:decomp}, it follows from
\cite[p. 294 and p. 303]{erd} that there exist $c\in\{0,1\}$ and $d\in k$ such that
$B$ (resp. $\hat{B}$) is Morita equivalent to $\Lambda_c=kQ/I_c$ (resp. 
$\hat{\Lambda}_d=kQ/\hat{I}_d$) as described in Figure \ref{fig:quiver}.
For the vertices $0,1$ in $Q$, the radical series of the corresponding
projective indecomposable $\Lambda_c$-modules $P_0,P_1$ and the corresponding
projective indecomposable  $\hat{\Lambda}_d$-modules $\hat{P}_0,\hat{P}_1$ are described 
in Figure \ref{fig:proj}. 
\begin{figure}[ht] \hrule \caption{\label{fig:quiver} 
The algebras 
$\Lambda_c=kQ/I_c$ ($c\in\{0,1\}$) and $\hat{\Lambda}_d=kQ/\hat{I}_d$ ($d\in k$).}
$$\xymatrix @R=-.2pc {
&0&1\\
Q=\quad& \ar@(ul,dl)_{\alpha} \bullet \ar@<.8ex>[r]^{\beta} &\bullet\ar@<.9ex>[l]^{\gamma}}$$
\begin{eqnarray*}
I_c&=&\langle \beta\gamma, \alpha^2-c(\gamma\beta\alpha)^2,(\gamma\beta\alpha)^2-
(\alpha\gamma\beta)^2\rangle,\\[1ex]
\hat{I}_d&=&\langle \gamma\beta\gamma-\alpha\gamma(\beta\alpha\gamma)^3,
\beta\gamma\beta-\beta\alpha(\gamma\beta\alpha)^3, \alpha^2-\gamma\beta(\alpha\gamma\beta)^3-
d(\alpha\gamma\beta)^4,\beta\alpha^2\rangle.
\end{eqnarray*}
\vspace{1ex}
\hrule
\end{figure}
\begin{figure}[ht] \hrule \caption{\label{fig:proj} The projective indecomposable
$\Lambda_c$-modules $P_0,P_1$ and the projective indecomposable $\hat{\Lambda}_d$-modules
$\hat{P}_0,\hat{P}_1$.}
$$P_0=\vcenter{ \xymatrix @R=.1pc @C=.3pc{&0&\\
0\ar@{.}[rddddd]&&1\\1&&0\\0&&0\\0&&1\\1&&0\\&0&}},\qquad 
P_1=\vcenter{ \xymatrix @R=.1pc @C=.3pc{1\\0\\0\\1\\0\\0\\1}};
\qquad\qquad \hat{P}_0=\vcenter{ \xymatrix @R=.1pc @C=.3pc{&0&\\
0\ar@{-}[rrdddddddddd]&&1\\1&&0\ar@{-}[llddddddddd]\\0&&0\\0&&1\\1&&0\\0&&0\\0&&1\\1&&0\\0&&0\\
0&&1\\1&&0\\&0&}},
\quad\hat{P}_1=\vcenter{ \xymatrix @R=.1pc @C=.3pc{&1&\\&0\ar@{-}[lddddd]&\\&&0\\&&1\\&&0\\&&0\\
1\ar@{-}[rddddd]&&1\\&&0\\&&0\\&&1\\&&0\\&0&\\&1&}}.$$
\hrule
\end{figure}

\begin{rem}
\label{rem:needthis}
Let $z$ be the non-trivial central element in $\hat{S}_5$ and let $Z=\langle z \rangle$
be the center of $\hat{S}_5$. In the following, we identify $S_5$ with $\hat{S}_5/Z$.
Let $\pi:k\hat{S}_5\to kS_5$ be the natural projection given by 
$\pi(g)=gZ$ for all $g\in \hat{S}_5$. Since $Z$ acts trivially on the simple $k\hat{S}_5$-modules,
we can identify the simple $kS_5$-modules with the simple $k\hat{S}_5$-modules.
This implies that the restriction of $\pi$ to $\hat{B}$ gives a surjective $k$-algebra
homomorphism $\pi_B:\hat{B}\to B$. In particular, if $V$ is a $kS_5$-module
belonging to $B$, then its inflation to $k\hat{S}_5$ via $\pi$ belongs to $\hat{B}$.
Let $\hat{e}$ be a sum of orthogonal primitive idempotents in $\hat{B}$ such that 
$\hat{e}\hat{B}\hat{e}$ is basic and Morita equivalent to $\hat{B}$, 
and let $e=\pi_B(\hat{e})$. Then $eBe$ is basic and Morita equivalent to $B$, and
the restriction of $\pi_B$ to $\hat{e}\hat{B}\hat{e}$ gives a surjective $k$-algebra homomorphism 
$\pi_e:\hat{e}\hat{B}\hat{e}\to eBe$.

If $c,d$ are such that $B$ is Morita equivalent to $\Lambda_c$ and $\hat{B}$ is Morita equivalent to
$\hat{\Lambda}_d$, let $\Lambda=\Lambda_c$ and $\hat{\Lambda}=\hat{\Lambda}_d$.
Then $\hat{e}\hat{B}\hat{e}\cong \hat{\Lambda}$ and
$eBe\cong \Lambda$, and $\pi_e$ induces a surjective $k$-algebra homomorphism
$\pi_\Lambda:\hat{\Lambda}\to \Lambda$.
It follows from the description of the projective indecomposable $\Lambda$-modules $P_0,P_1$
and the projective indecomposable $\hat{\Lambda}$-modules $\hat{P}_0,\hat{P}_1$ 
in Figure \ref{fig:proj} that $\Lambda\otimes_{\hat{\Lambda},\pi_\Lambda} \hat{P}_i\cong P_i$
for $i\in\{0,1\}$. In other words, the simple $\hat{\Lambda}$-module $\hat{P}_i/
\mathrm{rad}(\hat{P}_i)$ is isomorphic to the inflation via $\pi_\Lambda$ of the simple 
$\Lambda$-module $P_i/\mathrm{rad}(P_i)$ for $i\in\{0,1\}$.

Let $S_0=P_0/\mathrm{rad}(P_0)$ and $S_1=P_1/\mathrm{rad}(P_1)$.
Then $S_0$ corresponds to the trivial simple $kS_5$-module $T_0$, and
$S_1$ corresponds to the $4$-dimensional simple $kS_5$-module $T_1$ which is inflated 
from either one of the two $2$-dimensional simple $kA_5$-modules. 
The inflation of $T_0$ (resp. $T_1$) to $k\hat{S}_5$ via $\pi$ corresponds to the inflation
of $S_0$ (resp. $S_1$) to $\hat{\Lambda}$ via $\pi_\Lambda$. In particular, the former
inflations are simple $k\hat{S}_5$-modules, which we again denote by $T_0$ and $T_1$,
and the latter inflations
are simple $\hat{\Lambda}$-modules, which we again denote by $S_0$ and $S_1$. 
\end{rem}

We are now ready to prove part (a) of Theorem \ref{thm:supermain}. We assume the above
notation.

\medskip

\textit{Proof of part $(a)$ of Theorem $\ref{thm:supermain}$.}
Let $V$ be an indecomposable $kS_5$-module belonging to $B$, and denote its inflation
to $k\hat{S}_5$ also by $V$.  By Higman's criterion 
(see \cite[Thm. 1]{higman}), the $k\hat{S}_5$-module endomorphisms of $V$ that
factor through projective $k\hat{S}_5$-modules are precisely those in the image of the trace map
$\mathrm{Tr}_1^{\hat{S}_5}:\mathrm{End}_k(V)\to \mathrm{End}_{k\hat{S}_5}(V)$, where 
$\mathrm{Tr}_1^{\hat{S}_5}(\psi)(v)=\sum_{g\in \hat{S}_5}g\,\psi(g^{-1}v)$ for all 
$\psi\in\mathrm{End}_k(V)$
and all $v\in V$. Because $Z$ acts trivially on $V$, $\mathrm{Tr}_1^Z$ is multiplication by 
$2$. Hence $\mathrm{Tr}_1^Z$ is zero, which implies that 
$\mathrm{Tr}_1^{\hat{S}_5}=\mathrm{Tr}_Z^{\hat{S}_5}\circ \mathrm{Tr}_1^Z$ is also zero. 
It follows that
$\underline{\mathrm{End}}_{k\hat{S}_5}(V)\cong \mathrm{End}_{kS_5}(V)$. 
In particular, 
$\underline{\mathrm{End}}_{k\hat{S}_5}(V)\cong k$ if and only if
$\mathrm{End}_{kS_5}(V)\cong k$. 

Suppose now that $\mathrm{End}_{kS_5}(V)\cong k$. Then $V$ corresponds under the Morita
equivalence between $B$ and $\Lambda$ to an indecomposable $\Lambda$-module $M$
whose endomorphism ring is isomorphic to $k$. It follows from the description of the projective
indecomposable $\Lambda$-modules $P_0$ and $P_1$ in Figure \ref{fig:proj} that $M$
cannot be projective. Therefore, $M$ is inflated from an indecomposable
$\Lambda/\mathrm{soc}(\Lambda)$-module whose endomorphism ring is isomorphic to $k$.
Since $\Lambda/\mathrm{soc}(\Lambda)$
is a string algebra, all its indecomposable modules can be described using strings and
bands (see for example \cite{buri}). It follows that
a complete list of isomorphism classes of $\Lambda$-modules whose 
endomorphism rings are isomorphic to $k$ is given by
the following 6 uniserial $\Lambda$-modules which are uniquely determined, 
up to isomorphism, by their descending radical series:
\begin{equation}
\label{eq:list}
S_0, S_1, \begin{array}{c}0\\1\end{array}, \begin{array}{c}1\\0\end{array},
\begin{array}{c}0\\0\\1\end{array}, \begin{array}{c}1\\0\\0\end{array}.
\end{equation}
This completes the proof of part (a) of Theorem \ref{thm:supermain}.
\hfill $\Box$

\begin{rem}
\label{rem:needthis!!}
Since $\mathrm{Ext}^1_{\hat{\Lambda}}(S_0,S_1)\cong k \cong
\mathrm{Ext}^1_{\hat{\Lambda}}(S_1,S_0)$, there is up to isomorphism a unique
uniserial $\hat{\Lambda}$-module with descending composition factors $(S_0,S_1)$
(resp. $(S_1,S_0)$), which we denote by $M_{01}$ (resp. $M_{10}$). It follows that
the inflations via $\pi_\Lambda$ of the two-dimensional $\Lambda$-modules in the list 
(\ref{eq:list}) are isomorphic to $M_{01}$ or $M_{10}$. 

Because $\mathrm{Ext}^1_{\hat{\Lambda}}(S_0,M_{01})\cong k\cong 
\mathrm{Ext}^1_{\hat{\Lambda}}(M_{10},S_0)$, there is up to isomorphism a unique
uniserial $\hat{\Lambda}$-module with descending composition factors $(S_0,S_0,S_1)$
(resp. $(S_1,S_0,S_0)$), which we denote by $M_{001}$ (resp. $M_{100}$). It follows that
the inflations via $\pi_\Lambda$ of the three-dimensional $\Lambda$-modules in the list 
(\ref{eq:list}) are isomorphic to $M_{001}$ or $M_{100}$. 
\end{rem}


\section{Universal deformation rings}
\label{s:udr}

In this section we prove part (b) of Theorem \ref{thm:supermain}. 
We assume the notation from Section 
\ref{s:ks5}. In particular, $k$ is an algebraically closed field of characteristic $2$, and
$B$ (resp. $\hat{B}$) is the principal block of $kS_5$ (resp. of $k\hat{S}_5$). 
We need the following lemma. 

\begin{lemma}
\label{lem:ext2doubles5}
Suppose $\Lambda$, $\hat{\Lambda}$ and $\pi_\Lambda$ are as in Remark $\ref{rem:needthis}$.
Let $M$ be one of the two uniserial $\hat{\Lambda}$-modules $M_{001}$ or $M_{100}$ 
defined in Remark $\ref{rem:needthis!!}$.
Then $\mathrm{Ext}^2_{\hat{\Lambda}}(M,M)\cong k$.
\end{lemma}

\begin{proof}
It follows from the description of the projective indecomposable 
$\hat{\Lambda}$-mod\-ules in Figure \ref{fig:proj} that $\Omega^2_{\hat{\Lambda}}(M_{001})$
and $\Omega^{-2}_{\hat{\Lambda}}(M_{100})$ can be described as in Figure \ref{fig:syzygy}.
\begin{figure}[ht] \hrule \caption{\label{fig:syzygy} The syzygies $\Omega^2_{\hat{\Lambda}}(M_{001})$
and $\Omega^{-2}_{\hat{\Lambda}}(M_{100})$.}
$$\Omega^2_{\hat{\Lambda}}\left(M_{001}\right) = 
\vcenter{ \xymatrix @R=.1pc @C=.3pc{
1&&\\0\ar@{-}[rrddddddddd]&&\\0&&\\1&&\\0&&\\0&&\\1&&\\
0&&&&1\\0&&&&0\\1&&&0&&1\\0&&1\\&0&}}\quad\mbox{and}\quad
\Omega^{-2}_{\hat{\Lambda}}\left(M_{100}\right) = 
\vcenter{ \xymatrix @R=.1pc @C=.3pc{&&&&0&\\
&&&1&&0\\1&&0\ar@{-}[rrrddddddddd]&&&1\\&0&&&&0\\&1&&&&0\\&&&&&1\\&&&&&0\\
&&&&&0\\&&&&&1\\&&&&&0\\&&&&&0\\&&&&&1}}.
$$
\hrule
\end{figure}
This implies that 
$$\mathrm{Hom}_{\hat{\Lambda}}(\Omega^2_{\hat{\Lambda}}(M_{001}),M_{001})\cong k^2\quad
\mbox{ and }\quad
\mathrm{Hom}_{\hat{\Lambda}}(M_{100},\Omega^{-2}_{\hat{\Lambda}}(M_{100}))\cong k^2.$$
Since in both cases there is a one-dimensional subspace of these Hom spaces consisting of
homomorphisms factoring through $\hat{P}_0$, we obtain
$$\mathrm{Ext}^2_{\hat{\Lambda}}(M_{001},M_{001})\cong
\underline{\mathrm{Hom}}_{\hat{\Lambda}}(\Omega^2_{\hat{\Lambda}}(M_{001}),
M_{001})\cong k$$ and
$$\mathrm{Ext}^2_{\hat{\Lambda}}(M_{100},M_{100})\cong
\underline{\mathrm{Hom}}_{\hat{\Lambda}}(M_{100},\Omega^{-2}_{\hat{\Lambda}}(M_{100}))
\cong k.$$
\end{proof}

\medskip

\noindent
\textit{Proof of part $(b)$ of Theorem $\ref{thm:supermain}$.}
We go through the four different cases in the statement of the theorem.

\medskip

\noindent\textit{Case $(i).$}
Since the maximal abelian $2$-quotient group of both $S_5$ and $\hat{S}_5$ is a cyclic group of order $2$,
it follows e.g. from \cite[\S1.4]{maz1} that $R(S_5,T_0)\cong W[\mathbb{Z}/2]\cong R(\hat{S}_5,T_0)$.

\medskip

\noindent\textit{Case $(ii).$}
Let $E$ be one of the two non-isomorphic $2$-dimensional simple $kA_5$-modules, where $A_5$
denotes the alternating group on $5$ letters which is a subgroup of $S_5$. Then  
$T_1$ is isomorphic to the induction $\mathrm{Ind}_{A_5}^{S_5} E$. 
It follows from the description of the projective indecomposable $\Lambda$-modules
(resp. $\hat{\Lambda}$-modules) in Figure \ref{fig:proj} that 
$\mathrm{Ext}^1_{kS_5}(T_1,T_1)=0=\mathrm{Ext}^1_{k\hat{S}_5}(T_1,T_1)$.
Hence by \cite[Prop. 2.1.3]{3sim} and by \cite[Lemma 3.5(c)]{bl}, we have $R(S_5, T_1)\cong k$. 
Since it can be seen from the decomposition matrix for $\hat{B}$ in Figure \ref{fig:decomp} that 
$T_1$ when viewed as a $k\hat{S}_5$-module has a lift over $W$, we have
$R(\hat{S}_5,T_1)\cong W$.

\medskip

\noindent\textit{Case $(iii).$}
Suppose $V\in\left\{\begin{array}{c}T_0\\T_1\end{array},\begin{array}{c}T_1\\T_0\end{array}
\right\}$. It follows from the description of the projective indecomposable $\Lambda$-modules
in Figure \ref{fig:proj} that 
$$\mathrm{Ext}^1_{kS_5}(V,V)\cong
\underline{\mathrm{Hom}}_B(\Omega_B(V),V)\cong k$$ 
where 
$\Omega_B$ denotes the syzygy in the category of finitely generated $B$-modules.
Moreover, there is a non-split short exact sequence of $kS_5$-modules $0\to V\to U\to V\to 0$
where
\begin{equation}
\label{eq:sessos}
U=\begin{array}{c@{}c@{}c@{}c}&&T_0\\&T_0&&T_1\\T_1\end{array}
\mbox{ if $V=\begin{array}{c}T_0\\T_1\end{array}$, and}\quad
U=\begin{array}{c@{}c@{}c@{}c}T_1\\&T_0&&T_1\\&&T_0\end{array}
\mbox{ if $V=\begin{array}{c}T_1\\T_0\end{array}$.}
\end{equation} 
Let $C$ be the cyclic subgroup of $S_5$ of order $2$ generated by the transposition $(1,2)$.
Since $T_1\cong \mathrm{Ind}_{A_5}^{S_5} E$, it follows that $\mathrm{Res}_C^{S_5} T_1$ is
a projective $kC$-module, and hence isomorphic to $kC\oplus kC$. 
Moreover, if $T_{00}$ is the uniserial $kS_5$-module $T_{00}=
\begin{array}{c}T_0\\T_0\end{array}$ then $\mathrm{Res}_C^{S_5} T_{00}$
cannot be trivial since $\mathrm{Res}_{A_5}^{S_5} T_{00}$ is trivial. Hence $\mathrm{Res}_C^{S_5} 
T_{00}\cong kC$.
This means that 
$$\mathrm{Res}_C^{S_5}V\cong k\oplus (kC)^2,
\qquad\mbox{and}\qquad \mathrm{Res}_C^{S_5} U\cong (kC)^5.$$
Thus $\mathrm{Res}_C^{S_5}V$ is a $kC$-module whose stable endomorphism ring is isomorphic to
$k$ and whose universal deformation ring is $R(C,\mathrm{Res}_C^{S_5}V )\cong W[\mathbb{Z}/2]$.
Let $(U_{V,C},\phi_{U,C})$  be a universal lift of $\mathrm{Res}^{S_5}_{C}V$ over $W[\mathbb{Z}/2]$,
and let $(U_{V},\phi_U)$ be a universal lift of $V$ over $R(S_5,V)$. 
Then there exists a unique $W$-algebra homomorphism $\sigma:W[\mathbb{Z}/2]\to R(S_5,V)$ in 
$\mathcal{C}$  such that $(\mathrm{Res}^{S_5}_{C}U_{V},\mathrm{Res}^{S_5}_{C}\phi_U)$
is isomorphic to $(R(S_5,V)\otimes_{W[\mathbb{Z}/2],\sigma} U_{V,C},\phi_{U,C})$. 
To prove that $\sigma$ is surjective, consider all morphisms $\rho:R(S_5,V)\to k[\epsilon]/(\epsilon^2)$.
Since $\mathrm{Res}^{S_5}_C U\cong (kC)^5$ for the $kS_5$-module $U$ from 
$(\ref{eq:sessos})$, $\mathrm{Res}^{S_5}_C U$ defines a non-trivial lift of $\mathrm{Res}^{S_5}_C V$
over $k[\epsilon]/(\epsilon^2)$. Because $U$  defines a non-trivial lift of
$V$ over $k[\epsilon]/(\epsilon^2)$ and because $\mathrm{Ext}^1_{kS_5}(V,V)\cong k$, this implies
that as $\rho$ ranges over the morphisms 
$R(S_5,V)\to k[\epsilon]/(\epsilon^2)$, $\rho\circ \sigma$ ranges over the morphisms
$W[\mathbb{Z}/2]\to k[\epsilon]/(\epsilon^2)$. Hence $\sigma$ is surjective.
It follows from the decomposition matrix for $B$ in Figure \ref{fig:decomp}
and \cite[Prop. (23.7)]{CR} that $V$ has two non-isomorphic lifts over $W$ whose $F$-characters are
$\chi_3$ and $\chi_4$, respectively.
Thus there are two distinct morphisms $R(S_5,V)\to W$ in $\mathcal{C}$, which implies
that $\mathrm{Spec}(R(S_5,V))$ contains both points of the generic 
fiber of $\mathrm{Spec}(W[\mathbb{Z}/2])$.
Since the Zariski closure of these points is all of $\mathrm{Spec}(W[\mathbb{Z}/2])$,
it follows that $R(S_5,V)$ is isomorphic to $W[\mathbb{Z}/2]$.

Viewing $V$ as $k\hat{S}_5$-module by inflation, it follows from the description of the projective
indecomposable $\hat{\Lambda}$-modules in Figure \ref{fig:proj} that 
$$\mathrm{Ext}^1_{k\hat{S}_5}(V,V)\cong\underline{\mathrm{Hom}}_{\hat{B}}(\Omega_{\hat{B}}(V),
V)\cong k$$ 
where $\Omega_{\hat{B}}$ denotes the syzygy in the category of finitely generated 
$\hat{B}$-modules. Moreover, the module $U$ from $(\ref{eq:sessos})$ when viewed as a
$k\hat{S}_5$-module by inflation defines a non-trivial lift $(U,\nu)$ of $V$ over $k[\epsilon]/(\epsilon^2)$
when viewed as a $k\hat{S}_5$-module. Hence there exists a surjective $k$-algebra homomorphism
$$\tau:R(\hat{S}_5,V)/2R(\hat{S}_5,V)\to k[t]/(t^2)$$
corresponding to $(U,\nu)$. We now show that $\tau$ is a $k$-algebra isomorphism. 
Suppose this is false. Then there exists a surjective $k$-algebra homomorphism 
$\tau_1:R(\hat{S}_5,V)/2R(\hat{S}_5,V)\to k[t]/(t^3)$ such that $\delta\circ\tau_1=\tau$ where 
$\delta:k[t]/(t^3)\to k[t]/(t^2)$ is the natural projection. Let $(U_1,\nu_1)$ be a lift of $V$ over 
$k[t]/(t^3)$ relative to $\tau_1$. Then $k[t]/(t^2)\otimes_{k[t]/(t^3),\delta}U_1\cong U$ and 
$t^2U_1\cong V$. Thus we have a short exact sequence of $k[t]/(t^3)\,\hat{S}_5$-modules
\begin{equation}
\label{eq:thesos}
0\to t^2U_1\to U_1\to U\to 0.
\end{equation}
Since $\mathrm{Ext}^1_{k\hat{S}_5}(U,V)=0$, the sequence $(\ref{eq:thesos})$ splits as a sequence of 
$k\hat{S}_5$-modules. Thus $U_1\cong V\oplus U$ as $k\hat{S}_5$-modules. 
Since $V$ and $U$ are $kS_5$-modules, $U_1$ is inflated from a $kS_5$-module. 
Because there is no lift of $V$ over $k[t]/(t^3)$ when $V$ is viewed as a $kS_5$-module, this implies
that $U_1$ does not exist. Hence $\tau$ is a $k$-algebra isomorphism and
$R(\hat{S}_5,V)/2R(\hat{S}_5,V)\cong k[t]/(t^2)\cong R(S_5,V)/2R(S_5,V)$. Since
$R(S_5,V)$ is a $W$-algebra quotient of $R(\hat{S}_5,V)$ which is free as a $W$-module, 
this implies that
$R(\hat{S}_5,V)\cong R(S_5,V)\cong W[\mathbb{Z}/2]$.

\medskip

\noindent\textit{Case $(iv).$}
Suppose $V\in\left\{\begin{array}{c}T_0\\T_0\\T_1\end{array},
\begin{array}{c}T_1\\T_0\\ T_0\end{array}\right\}$. It follows from the description of the projective 
indecomposable modules in Figure \ref{fig:proj} that 
$$\mathrm{Ext}^1_{kS_5}(V,V)\cong k\cong \mathrm{Ext}^1_{k\hat{S}_5}(V,V).$$
Moreover, we see from Figure \ref{fig:proj} that there is a uniserial $kS_5$-module $X$ with
descending composition factors
$$(T_0,T_0,T_1,T_0,T_0,T_1)\qquad \mbox{(resp. $(T_1,T_0,T_0,T_1,T_0,T_0)$)}$$
such that $X$ defines a lift $(X,\xi)$ of $V$ over $k[t]/(t^2)$ when the descending composition 
factors of $V$ are $(T_0,T_0,T_1)$
(resp. $(T_1,T_0,T_0)$). Additionally, there is a uniserial $k\hat{S}_5$-module
$Y$ with descending composition factors
$$(T_0,T_0,T_1,T_0,T_0,T_1,T_0,T_0,T_1)\qquad \mbox{(resp. 
$(T_1,T_0,T_0,T_1,T_0,T_0,T_1,T_0,T_0)$)}$$ 
such that $Y$ defines a lift $(Y,\zeta)$ of $V$ over $k[t]/(t^3)$
when $V$ is viewed as a $k\hat{S}_5$-module by inflation. Since
$$\mathrm{Ext}^1_{kS_5}(X,V)=0=\mathrm{Ext}^1_{k\hat{S}_5}(Y,V),$$ we see that
$$R(S_5,V)/2R(S_5,V)\cong k[t]/(t^2)\qquad\mbox{and}\qquad 
R(\hat{S}_5,V)/2R(\hat{S}_5,V)\cong k[t]/(t^3).$$
Moreover, the isomorphism class of the lift $(X,\xi)$ is the universal mod $2$ deformation of $V$ 
when $V$ is viewed as a $kS_5$-module,
and the isomorphism class of the lift $(Y,\zeta)$ is the universal mod $2$ deformation of $V$ 
when $V$ is viewed as a $k\hat{S}_5$-module.

It follows from the decomposition matrix for $B$ in Figure \ref{fig:decomp} that $V$ has a 
lift over $W$. Hence by \cite[Lemma 2.1]{bc5}, there exist $\mu\in\{0,1\}$, $m\in\mathbb{Z}^+$ and 
$\lambda\in W$ such that 
$$R(S_5,V)\cong W[[t]]/(t^2-2\lambda t,\mu 2^mt).$$
Since $X\cong \Omega_B^i(T_1)$ for either $i=1$ or $i=-1$, it follows that
$X$ has a universal deformation ring when viewed as a $kS_5$-module and
$R(S_5,X)\cong k$ by the proof of Case (ii) and by Lemma \ref{lem:defhelp}.
If $\mu=0$ (resp. $\mu=1$), then $R(S_5,V)$ (resp. $(W/2^mW)\otimes_W R(S_5,V)$) 
is free over $W$ (resp. $W/2^mW$). This implies that $X$, when regarded as a $kS_5$-module,
has a lift over $W$ (resp. $W/2^mW$). Hence $\mu=1$ and $m=1$, and so $R(S_5,V)\cong
W[[t]]/(t^2-2\lambda t,2t)\cong W[[t]]/(t^2,2t)$.

Since $\mathrm{Ext}^2_{k\hat{S}_5}(V,V)\cong k$
by Lemma \ref{lem:ext2doubles5}, it follows from \cite[\S1.6]{maz1}
that there exists an element $f(t)\in W[[t]]$ such that $R(\hat{S}_5,V)\cong W[[t]]/(f(t))$. 
Since $R(\hat{S}_5,V)/2 R(\hat{S}_5,V)\cong k[t]/(t^3)$, it follows by the Weierstrass Preparation 
Theorem (see e.g. \cite[Thm. IV.9.2]{lang}) that $f(t)$ can be taken to be of the form 
$f(t)=t^3+at^2+bt+c$ for certain $a,b,c\in 2W$. In particular, $R(\hat{S}_5,V)$ is free as a $W$-module.
Let $(Y^W,\zeta_W)$ be a universal lift of $V$ over $R(\hat{S}_5,V)$
when $V$ is viewed as a $k\hat{S}_5$-module. 
Since the isomorphism class of $(Y,\zeta)$ is the universal mod $2$ deformation of $V$ as a
$k\hat{S}_5$-module, it follows that $Y^W$ defines a lift $(Y^W,\omega)$ of $Y$ over $W$
when $Y$ is viewed as a $k\hat{S}_5$-module. 
If $Y/\mathrm{rad}(Y)\cong T_1$ then $Y$ is a quotient module of the projective indecomposable
$k\hat{S}_5$-module $\hat{P}_{T_1}$ with $\hat{P}_{T_1}/\mathrm{rad}(\hat{P}_{T_1})\cong T_1$. 
Hence $Y^W$ must be a quotient module of 
the projective indecomposable $W\hat{S}_5$-module $\hat{P}^W_{T_1}$ which is a lift of
$\hat{P}_{T_1}$ over $W$, and we define $Z^W=Y^W$.
If $\mathrm{soc}(Y)\cong T_1$ then $\Omega^{-1}(Y)$ is a quotient module of $\hat{P}_{T_1}$, and
by Lemma \ref{lem:defhelp}, $\Omega^{-1}(Y)$ has a lift $({Y'}^W,\omega')$ over $W$. 
Hence ${Y'}^W$ must be a
quotient module of $\hat{P}^W_{T_1}$. But then the kernel of the surjection $\hat{P}^W_{T_1}\to {Y'}^W$
is a $W$-pure submodule of $\hat{P}^W_{T_1}$, and we define
$Z^W$ to be this kernel.
Therefore we have for both cases of $Y$ that $Z^W$ defines a lift of $Y$ over $W$ and
that $Z^W$ is either a quotient 
module or a submodule of $\hat{P}^W_{T_1}$.
Thus it follows from the decomposition matrix for $\hat{B}$ in Figure
 \ref{fig:decomp} that the $F$-character of $Z^W$ is equal to 
$$\chi_Z=\psi_6+(\psi_7+\psi_8).$$
This implies by Remark \ref{rem:ordinary} that the endomorphism ring of
$F\otimes_W Z^W\cong V_6\oplus V_{78}$ is isomorphic to $F\times F(\sqrt{2})$.
Let $u$ be an element in $\hat{S}_5$ of order $8$ belonging to the conjugacy class
$C_9$ in Figure \ref{fig:ordchar2s5} and let $K_u$ be its class sum in 
$W\hat{S}_5$. Because $K_u$ lies in the center of $W\hat{S}_5$,  multiplication by
$K_u$ defines a $W\hat{S}_5$-module endomorphism $\kappa_u$ of $Z^W$. 
Since $Z^W$ is free as a $W$-module, the endomorphism ring
$\mathrm{End}_{W\hat{S}_5}(Z^W)$ embeds naturally into 
\begin{eqnarray*}
F\otimes_W\mathrm{End}_{W\hat{S}_5}(Z^W)&\cong& 
\mathrm{End}_{F\hat{S}_5}(F\otimes_WZ^W)\\
&\cong& \mathrm{End}_{F\hat{S}_5}(V_6)\times \mathrm{End}_{F\hat{S}_5}(V_{78})
\;\cong\, F\times F(\sqrt{2}).
\end{eqnarray*}
Hence $\kappa_u$ corresponds to an element in $F\times F(\sqrt{2})$ which we can read off 
from the ordinary character table of $\hat{S}_5$. Namely,
the endomorphism $\kappa_u$ in $\mathrm{End}_{W\hat{S}_5}(U^W)$ corresponds 
to the element 
$$(0,5\,\sqrt{2})\in F\times F(\sqrt{2}).$$ 
Because $(0,5\,\sqrt{2})$ generates a
$W$-subalgebra of $F\times F(\sqrt{2})$ which is isomorphic to $W[[t]]/(t^3-2t)$,
it follows that $Z^W$ is a $W[[t]]/(t^3-2t)\hat{S}_5$-module. Taking a $k[t]/(t^3)$-basis 
$\{b_1,\ldots,b_6\}$ of $Y$, we can lift this basis to a  subset $\{c_1,\ldots, c_6\}$ of $Z^W$ 
which generates $Z^W$ as a $W[[t]]/(t^3-2t)$-module. 
Since $F\otimes_W Z^W$ is a free $(F\times F(\sqrt{2}))$-module 
of rank $6$, it follows that $c_1,\ldots, c_6$ must be linearly independent over $W[[t]]/(t^3-2t)$.
Thus $Z^W$ defines a lift of $V$ over $W[[t]]/(t^3-2t)$. Since 
$Z^W/2Z^W\cong Y$ is an indecomposable $k\hat{S}_5$-module, 
$W[[t]]/(t^3-2t)$ is a quotient algebra of $R(\hat{S}_5,V)$.
This implies that we can take $f(t)=t^3-2t$, and hence $R(\hat{S}_5,V)\cong W[[t]]/(t^3-2t)$.
\hfill $\Box$

\begin{rem}
\label{rem:requested}
Suppose $G$ and $\hat{G}$ are two finite groups such that the Sylow $2$-subgroups
of $G$ are dihedral groups of order 8 and the Sylow $2$-subgroups of $\hat{G}$ are
generalized quaternion groups of order 16 and such that $\hat{G}$ is an extension
of $G$ by a central subgroup of order 2. Moreover, assume that there exist $c\in\{0,1\}$ and
$d\in k$ such that the principal block $B$ (resp. $\hat{B}$) of $kG$ (resp. $k\hat{G}$) is Morita 
equivalent to $\Lambda_c$ (resp. $\hat{\Lambda}_d$) as in Figure \ref{fig:quiver}. 
Many of the arguments in this paper work for this more general case. However, when computing
the universal deformation rings for the cases (iii) and (iv) of part (b) of Theorem \ref{thm:supermain}, one
runs into the following issues. First, one needs to prove in general that there is an element of
order $2$ in $G$ that can take the place of the transposition $(1,2)\in S_5$ when computing the
universal deformation ring $R(G,V)$ in case (iii). Second, one needs to establish similar
facts to the ones in Remark \ref{rem:ordinary} for the irreducible representations of $\hat{G}$
over $F$ which belong to $\hat{B}$, including the values of the ordinary characters on certain
conjugacy classes, when computing the universal deformation ring $R(\hat{G},V)$ in case (iv).
\end{rem}


\section{Appendix: The ordinary and the $2$-modular character table of $\hat{S}_5$}
\label{s:append}

The ordinary character table of $\hat{S}_5$ can be found for example in \cite[p. 289]{hoffhum}.
It is then straightforward to determine the ordinary character table of $S_5$ and also the
$2$-modular character table of $S_5$ and $\hat{S}_5$. The ordinary characters
$\chi_1,\ldots, \chi_4, \chi_5$ of $S_5$ in Figure \ref{fig:decomp} correspond to the
ordinary characters $\psi_1,\ldots,\psi_4,\psi_6$ of $\hat{S}_5$ in Figures \ref{fig:decomp}
and \ref{fig:ordchar2s5}.

\begin{figure}[ht]  \caption{\label{fig:ordchar2s5} The ordinary character table of
$\hat{S}_5$.}
\begin{small}
$$\begin{array}{lrrrrrrrrrrrr}
\mathrm{class:}&C_1&C_2&C_3&C_4&C_5&C_6&C_7&C_8&C_9&C_{10}&C_{11}&C_{12}\\
\mathrm{order:}&1&2&4&3&6&5&10&4&8&8&12&12\\
\mathrm{length:}&1&1&30&20&20&24&24&20&30&30&20&20\\[2ex]
\psi_1&1&1&1&1&1&1&1&1&1&1&1&1\\
\psi_2&1&1&1&1&1&1&1&-1&-1&-1&-1&-1\\
\psi_3&5&5&1&-1&-1&0&0&1&-1&-1&1&1\\
\psi_4&5&5&1&-1&-1&0&0&-1&1&1&-1&-1\\
\psi_5&4&-4&0&-2&1&-1&1&0&0&0&0&0\\
\psi_6&6&6&-2&0&0&1&1&0&0&0&0&0\\
\psi_7&6&-6&0&0&0&1&-1&0&\sqrt{2}&-\sqrt{2}&0&0\\
\psi_8&6&-6&0&0&0&1&-1&0&-\sqrt{2}&\sqrt{2}&0&0\\
\psi_9&4&4&0&1&1&-1&-1&2&0&0&-1&-1\\
\psi_{10}&4&4&0&1&1&-1&-1&-2&0&0&1&1\\
\psi_{11}&4&-4&0&1&-1&-1&1&0&0&0&\sqrt{3}&-\sqrt{3}\\
\psi_{12}&4&-4&0&1&-1&-1&1&0&0&0&-\sqrt{3}&\sqrt{3}
\end{array}$$
\end{small}
\end{figure}

\begin{figure}[ht]  \caption{\label{fig:modchar2s5} The $2$-modular character table of
$S_5$ and $\hat{S}_5$.}
\begin{small}
$$\begin{array}{lrrrrrrrrrrrr}
\mathrm{class:}&C_1&C_4&C_6\\
\mathrm{order:}&1&3&5\\[2ex]
\varphi_0&1&1&1\\
\varphi_1&4&-2&-1\\
\varphi_2&4&1&-1\end{array}$$
\end{small}
\end{figure}


\end{document}